\documentclass[11pt,leqno]{article}
\usepackage[margin=1in]{geometry} 
\usepackage{amssymb,amsfonts,amsmath,bbm,mathrsfs,stmaryrd,mathtools}
\usepackage{xcolor}
\usepackage{url}

\usepackage{extarrows}

\usepackage[shortlabels]{enumitem}
\usepackage{tensor}

\usepackage{xr}
\usepackage[T1]{fontenc}

\usepackage[colorlinks,
linkcolor=black!75!red,
citecolor=blue,
pdftitle={},
pdfproducer={pdfLaTeX},
pdfpagemode=None,
bookmarksopen=true,
bookmarksnumbered=true,
backref=page]{hyperref}

\usepackage{tikz}
\usetikzlibrary{arrows,calc,decorations.pathreplacing,decorations.markings,intersections,shapes.geometric,through,fit,shapes.symbols,positioning,decorations.pathmorphing}

\usepackage{braket}

\usepackage[amsmath,thmmarks,hyperref]{ntheorem}
\usepackage{cleveref}

\newcommand\nthalias[1]{\AddToHook{env/#1/begin}{\crefalias{lemma}{#1}}}

\nthalias{definition}
\nthalias{example}
\nthalias{examples}
\nthalias{remark}
\nthalias{remarks}
\nthalias{convention}
\nthalias{notation}
\nthalias{construction}
\nthalias{theoremN}
\nthalias{propositionN}
\nthalias{corollaryN}
\nthalias{lemma}
\nthalias{proposition}
\nthalias{corollary}
\nthalias{theorem}
\nthalias{conjecture}
\nthalias{question}
\nthalias{assumption}

\creflabelformat{enumi}{#2#1#3}

\crefname{section}{Section}{Sections}
\crefformat{section}{#2Section~#1#3} 
\Crefformat{section}{#2Section~#1#3} 

\crefname{subsection}{\S}{\S\S}
\AtBeginDocument{%
  \crefformat{subsection}{#2\S#1#3}%
  \Crefformat{subsection}{#2\S#1#3}%
}

\crefname{subsubsection}{\S}{\S\S}
\AtBeginDocument{%
  \crefformat{subsubsection}{#2\S#1#3}%
  \Crefformat{subsubsection}{#2\S#1#3}%
}

\theoremstyle{plain}

\newtheorem{lemma}{Lemma}[section]

\newtheorem{corollary}[lemma]{Corollary}
\newtheorem{theorem}[lemma]{Theorem}

\newtheorem{question}[lemma]{Question}

\theoremstyle{plain}
\theoremnumbering{Alph}

\theoremstyle{plain}
\theorembodyfont{\upshape}
\theoremsymbol{\ensuremath{\blacklozenge}}

\newtheorem{example}[lemma]{Example}

\newtheorem{remark}[lemma]{Remark}

\newtheorem{notation}[lemma]{Notation}

\crefname{definition}{definition}{definitions}
\crefformat{definition}{#2definition~#1#3} 
\Crefformat{definition}{#2Definition~#1#3} 

\crefname{ex}{example}{examples}
\crefformat{example}{#2example~#1#3} 
\Crefformat{example}{#2Example~#1#3} 

\crefname{exs}{example}{examples}
\crefformat{examples}{#2example~#1#3} 
\Crefformat{examples}{#2Example~#1#3} 

\crefname{remark}{remark}{remarks}
\crefformat{remark}{#2remark~#1#3} 
\Crefformat{remark}{#2Remark~#1#3} 

\crefname{remarks}{remark}{remarks}
\crefformat{remarks}{#2remark~#1#3} 
\Crefformat{remarks}{#2Remark~#1#3} 

\crefname{convention}{convention}{conventions}
\crefformat{convention}{#2convention~#1#3} 
\Crefformat{convention}{#2Convention~#1#3} 

\crefname{notation}{notation}{notations}
\crefformat{notation}{#2notation~#1#3} 
\Crefformat{notation}{#2Notation~#1#3} 

\crefname{table}{table}{tables}
\crefformat{table}{#2table~#1#3} 
\Crefformat{table}{#2Table~#1#3}

\crefname{lemma}{lemma}{lemmas}
\crefformat{lemma}{#2lemma~#1#3} 
\Crefformat{lemma}{#2Lemma~#1#3} 

\crefname{proposition}{proposition}{propositions}
\crefformat{proposition}{#2proposition~#1#3} 
\Crefformat{proposition}{#2Proposition~#1#3} 

\crefname{propositionN}{proposition}{propositions}
\crefformat{propositionN}{#2proposition~#1#3} 
\Crefformat{propositionN}{#2Proposition~#1#3} 

\crefname{corollary}{corollary}{corollaries}
\crefformat{corollary}{#2corollary~#1#3} 
\Crefformat{corollary}{#2Corollary~#1#3} 

\crefname{corollaryN}{corollary}{corollaries}
\crefformat{corollaryN}{#2corollary~#1#3} 
\Crefformat{corollaryN}{#2Corollary~#1#3} 

\crefname{theorem}{theorem}{theorems}
\crefformat{theorem}{#2theorem~#1#3} 
\Crefformat{theorem}{#2Theorem~#1#3} 

\crefname{theoremN}{theorem}{theorems}
\crefformat{theoremN}{#2theorem~#1#3} 
\Crefformat{theoremN}{#2Theorem~#1#3} 

\crefname{enumi}{}{}
\crefformat{enumi}{#2#1#3}
\Crefformat{enumi}{#2#1#3}

\crefname{assumption}{assumption}{Assumptions}
\crefformat{assumption}{#2assumption~#1#3} 
\Crefformat{assumption}{#2Assumption~#1#3} 

\crefname{construction}{construction}{Constructions}
\crefformat{construction}{#2construction~#1#3} 
\Crefformat{construction}{#2Construction~#1#3} 

\crefname{question}{question}{Questions}
\crefformat{question}{#2question~#1#3} 
\Crefformat{question}{#2Question~#1#3} 

\crefname{equation}{}{}
\crefformat{equation}{(#2#1#3)} 
\Crefformat{equation}{(#2#1#3)}

\numberwithin{equation}{section}

\theoremstyle{nonumberplain}
\theoremsymbol{\ensuremath{\blacksquare}}

\newtheorem{proof}{Proof}
\newcommand\pf[1]{\newtheorem{#1}{Proof of \Cref{#1}}}

\newcommand\bC{{\mathbb C}}

\newcommand\bG{{\mathbb G}}

\newcommand\bP{{\mathbb P}}

\newcommand\bR{{\mathbb R}}

\newcommand\bZ{{\mathbb Z}}

\newcommand\cB{{\mathcal B}}

\newcommand\cD{{\mathcal D}}

\newcommand\cL{{\mathcal L}}
\newcommand\cM{{\mathcal M}}

\newcommand\cP{{\mathcal P}}

\newcommand\cS{{\mathcal S}}

\newcommand\cX{{\mathcal X}}

\DeclareMathOperator{\Hom}{\mathrm{Hom}}

\DeclareMathOperator{\spn}{\mathrm{span}}

\newcommand{\qedhere}{\mbox{}\hfill\ensuremath{\blacksquare}}

\title{Fano schemes of sub-maximal elementary symmetric functions}
\author{Alexandru Chirvasitu}

\begin{document}

\date{}

\newcommand{\Addresses}{{
  \bigskip
  \footnotesize

  \textsc{Department of Mathematics, University at Buffalo}
  \par\nopagebreak
  \textsc{Buffalo, NY 14260-2900, USA}  
  \par\nopagebreak
  \textit{E-mail address}: \texttt{achirvas@buffalo.edu}


}}

\maketitle

\begin{abstract}
  Denote by $E_r$ the $r^{th}$ elementary symmetric polynomial in $\dim V$ variables for a vector space $V$ over an infinite field $\Bbbk$. We describe the rational points on the Fano scheme $F_{d-1}(Z(E_{\dim V-1}))$ of projective $(d-1)$-spaces contained in the zero locus of $E_{\dim V-1}$. Isolated points exist precisely for $\dim V=2d$, in which case they are in bijection with the $1\cdot 3\cdots (2d-1)$ pairings on a $2d$-element set. This, in particular, confirming a conjecture of Ambartsoumian, Auel and Jebelli to the effect that (over $\mathbb{R}$) all isolated points are recoverable via integral star transforms with appropriate symbols.
\end{abstract}


\vspace{.5cm}

\noindent{MSC 2020: 14J45; 05E05; 14M15; 20C15; 13A50; 13F30; 12J20; 14N05
  
  
}


\section*{Introduction}

The present note is motivated by a number of questions raised in \cite{2507.22138v1}, in the context of studying integral/differential operators on $\bR^n$ and their attendant geometry.

Write $\bP V$ for the projective space attached to a vector space $V$ (i.e. the set of lines in $V$) and, following \cite[\S 3.2]{3264}, 
\begin{equation*}
  \bG(k,\dim V-1)
  =
  \bG(k,\bP V)
  =
  G(k+1,V)
  =
  G(k+1,\dim V)
\end{equation*}
for the \emph{Grassmannian} of projective-dimension-$k$-planes in $\bP V$ (equivalently, linear $(k+1)$-dimensional subspaces of $V$). Recall also \cite[\S 6.1.1]{3264} the \emph{Fano schemes}
\begin{equation}\label{eq:fano}
  F_k(X)
  :=
  \left\{
    W\in \bG(k,\bP V)
    \ :\
    W\subseteq X
  \right\}
  \subseteq
  \bG(k,\bP V)=G(k+1,V)
\end{equation}
attached to closed subvarieties $X\subseteq \bP V$. Assume a basis $(e_i)$ for $V$ fixed, providing coordinate functions $x_i$ and hence zero loci
\begin{equation*}
  X_r=X_{r,\dim V}:=Z(E_r)\subseteq \bP V
  ,\quad
  E_r:=\text{$r^{th}$ \emph{elementary symmetric polynomial} \cite[\S A.1]{fh_rep-th}}.
\end{equation*}
Working over the reals, \cite[Theorem 5.1]{2507.22138v1} constructs points of $F_{\bullet}(X_{m-r,m})$ in the following PDE-motivated fashion.

\begin{itemize}[wide]
\item Write $\cX_u$, $u\in \bR^n$ for the \emph{divergent beam transform} \cite[Definition 1]{MR4310170}
  \begin{equation*}
    f
    \xmapsto{\quad \cX_u\quad}
    \int_{-\infty}^0 f(x+tu)\ \mathrm{d}t
  \end{equation*}
  on compactly-supported (smooth, say) on $\bR^n$. 

\item Define the \emph{star transform} \cite[Definition 2]{MR4310170} $\cS=\cS_{E_r,(u_i)_i}$ attached to $E_r$ and an $m$-tuple $(u_i)_{i=1}^m\subset \bR^n$ as
  \begin{equation}\label{eq:s.trnsf.er}
    f
    \xmapsto{\quad \cS\quad}
   E_r\left(\cX_{u_1},\ \cdots,\ \cX_{u_m}\right) f.
  \end{equation}

\item The range of
  \begin{equation}\label{eq:u.as.op}
    \left(\bR^n\right)^*
    \xrightarrow{\quad (u_i)_{i=1}^m\quad}
    \bR^m
  \end{equation}
  (i.e. regarding the $u_i$ as rows of an $m\times n$ matrix) is then shown in \cite[Theorem 5.1]{2507.22138v1} to constitute a point of $F_{\dim\spn\left\{u_i\right\}-1}(X_{m-r,m})\subseteq \bG(\dim\spn\left\{u_i\right\}-1, \bP \left(\bR^m\right))$ whenever the transform $\cS$ in question is non-invertible.

\item \cite[Proposition 5.3]{2507.22138v1} moreover specializes that discussion to provide $1\cdot 3\cdots (2d-1)$ isolated (real) points on $F_{d-1}(X_{2d-1,2d})$.
\end{itemize}

It is in that context that it becomes natural to conjecture the list of isolated $F_{d-1}(X_{2d-1,2d})$-points in \cite[Proposition 5.3(1)]{2507.22138v1} exhaustive. \Cref{th:stt.cj} confirms that conjecture by describing the ($\Bbbk$-rational points on the) Fano schemes $F_{d-1}(X_{m-1,m})$ attached to immediately-sub-maximal elementary symmetric polynomials over arbitrary infinite fields. Some notation will help streamline the statement.

\begin{notation}\label{not:part.eq}
  \begin{enumerate}[(1),wide]

  \item For $m\in \bZ_{\ge 0}$ the symbol $\cP^{\text{condition}}_{[m]}$ denotes the set of partitions
    \begin{equation*}
      \pi\quad:\quad
      [m]:=\left\{1\cdots m\right\}
      =
      \bigsqcup_{1}^k
      \pi^i
      ,\quad
      \pi^i\subseteq [m]
    \end{equation*}
    whose part sizes $\left|\pi^i\right|$ satisfy the condition. Examples include $\cP^{\ge d}_{[m]}$ (set of partitions into cardinality-$(\ge d)$ parts), $\cP^{\text{even}}_{[m]}$ (that of partitions into even parts), plain $\cP_{[m]}$ (no constraints at all), etc.
    
    The number writing $\sharp \pi$ for the number $k$ of parts of $\pi=\left(\pi^i\right)_{i=1}^k\in \cP_{[m]}$.
    
  \item Having fixed a basis $(e_{\ell})_{\ell=1}^{\dim V}$ for a $\Bbbk$-vector space, a scalar tuple $\mathbf{c}\in \Bbbk^{\dim V}$ and a partition $\pi\in \cP^{\ge 1}_{[\dim V]}$ we write
    \begin{equation*}
      V_{\pi,\mathbf{c}}
      :=
      \left\{
        \sum_{i=1}^{\sharp\pi}t_i \left(\sum_{j\in \pi^i}c_j e_j\right)
        \ :\
        t_i\in \Bbbk
      \right\}
      \le
      V
    \end{equation*}
    (a $\sharp\pi$-dimensional subspace of $V$). 
  \end{enumerate}
\end{notation}

\begin{theorem}\label{th:stt.cj}
  Consider a vector space $V$ of dimension $m\in \bZ_{>0}$ over an infinite field $\Bbbk$ with a fixed basis $(e_j)_{j=1}^m$, setting
  \begin{equation*}
    \begin{aligned}
      X=X_{m-1}
      &:=
        \left\{
        x\in \bP V
        \ :\ 
        f(x)=0
        \right\}\\
      f=E_{m-1}
      &:=
        \text{$(m-1)^{st}$ elementary symmetric function in the $m$ coordinates attached to $(e_j)$}.
    \end{aligned}    
  \end{equation*}

  \begin{enumerate}[(1),wide]
  \item\label{item:th:stt.cj:gen} The $\Bbbk$-points of the Fano scheme $F_{d-1}(X)$, $d\in \bZ_{>0}$ are precisely
    \begin{enumerate}[(a),wide]
    \item\label{item:th:stt.cj:nis} the elements of the union
      \begin{equation*}
        \bigcup_{W} \bG(d-1,\bP W)
      \end{equation*}
      for $W\le V$ ranging over the $(m-2)$-dimensional zero sets of the $\tbinom{m}{2}$ pairs of coordinates;
      
    \item\label{item:th:stt.cj:is} and those of the form
      \begin{equation}\label{eq:pvpc}
        \bP V_{\pi,\mathbf{c}}\le \bP V
        \quad\text{for}\quad
        \left\{
          \begin{aligned}
            \pi
            \in
            \cP^{\ge 1}_{[m]}
            \quad&\text{and}\quad
                   \mathbf{c}
                   \in \left(\Bbbk^{\times}\right)^m\\
            \forall \left(1\le i\le \sharp \pi=d\right)
            \quad&:\quad
                   \sum_{j\in \pi^i}\frac 1{c_j}=0
          \end{aligned}
        \right.
      \end{equation}
    \end{enumerate}

  \item\label{item:th:stt.cj:d.lg} In particular, \Cref{item:th:stt.cj:nis} exhausts the possibilities if $2d>m$.

  \item\label{item:th:stt.cj:isol} Isolated points exist only for $2d=m$, in which case they are the projectivizations of the $d$-planes in $\bC^{2d}$ defined by $s$ equations
      \begin{equation*}
        \sum_{j\in p_{\alpha}}x_j=0
        ,\quad
        1\le \alpha\le d,
      \end{equation*}
      for $(p_{\alpha})_{\alpha}$ ranging over the partitions of $[2d]:=\left\{1\cdots 2d\right\}$ into $d$ pairs.         
  \end{enumerate}
\end{theorem}

We also record the following immediate consequence (verifying the conjectural description of the components of $F_{d-1}\left(X_{2d-1}\right)$ mentioned immediately following \cite[Proposition 5.3]{2507.22138v1}).

\begin{corollary}\label{cor:cmpnnt}
  In the context of \Cref{th:stt.cj}, working over an algebraically-closed field so as to identify varieties with their point-sets, the components of $F_{d-1}(X)$ are
  \begin{itemize}[wide]
  \item those of type \Cref{item:th:stt.cj:nis}: the Grassmannians $\bG(d-1,\bP W)$, one for each of the $\tbinom{m}{2}$ $(m-2)$-dimensional $W\le V$ obtained by annihilating pairs of coordinates;
    
  \item and those of type \Cref{item:th:stt.cj:is}, indexed by partitions: $\left\{\bP V_{\pi,\mathbf{c}}\right\}_{\mathbf{c}}$ for $\pi\in\cP^{\ge 1}_{[m]}$ with $d$ parts.  \qedhere
  \end{itemize}
\end{corollary}

\Cref{se:pol.inv} branches out to address a problem posed in \cite[\S 4]{2507.22138v1}, relating to star transforms \Cref{eq:s.trnsf.er} exhibiting symmetries under a linear action of a finite group $\bG$. Specifically:
\begin{itemize}[wide]
\item Define the differential operator
  \begin{equation*}
    f
    \xmapsto{\quad \cL=\cL_{E_r,(u_i)_i}\quad}
    E_{m-r}\left(\cD_{u_1},\ \cdots,\ \cD_{u_m}\right) f
  \end{equation*}
  \emph{dual} \cite[Definition 2.2]{2507.22138v1} to \Cref{eq:s.trnsf.er}, with the $\cD_{\bullet}$ denoting partial differential operators in the indicated directions.

\item $\cL$ has an associated \emph{symbol} (\cite[(5.44)]{grb_pdo}, \cite[Definition 3.3.13]{nar_real-cplx_2e_1985}) $\sigma_{\cL}(\xi)$, regarded as a polynomial on the same dual space $\left(\bR^n\right)^*$ displayed in \Cref{eq:u.as.op}.

\item Assuming \Cref{eq:u.as.op} to be a $\bG$-equivariant map for a representation of the finite group $\bG$ and a permutation representation carried by the codomain $\bR^m$, \cite[Theorem 3.2]{2507.22138v1} proves the symbol $\bG$-invariant.

\item Whereupon \cite[\S 4, Question]{2507.22138v1} asks (in one possible interpretation) whether the algebra of such $\bG$-invariant polynomials exhausts $\bR[\xi_i]^{\bG}$. 
\end{itemize}

\Cref{th:orbs.suff} below answers this in the affirmative, as a simple consequence of work constructing sufficiently many polynomial $\bG$-invariants via representation-theoretic analogues \cite[\S 2]{zbMATH01097431} of \emph{Chern classes} \cite[\S 14]{ms-cc}.  

\subsection*{Acknowledgments}

I am grateful for illuminating suggestions, tips and comments from G. Ambartsoumian, M. Fulger, M. J. Latifi and C. Raicu. 


\section{Fano schemes of almost-top elementary symmetric functions}\label{se:fano}

Recall the covering
\begin{equation*}
  \begin{aligned}
    G(k+1,V)
    &=
      \bigcup_{\substack{L\le V\\\dim L=\dim V-(K+1)}} U_L\\
    U_L
    &:=
      \left\{
      W\in G(k+1,V)
      \ :\
      W\cap L=\{0\}
      \right\}
  \end{aligned}  
\end{equation*}
of \cite[\S 3.2.2]{3264} by affine open patches: having fixed a decomposition $V=L'\oplus L$, we have an identification
\begin{equation*}
  \Hom(L',L)\ni T
  \xmapsto[\quad\cong\quad]{\quad}
  \left(
    \text{graph of $T$}
    \le
    L'\oplus L=V
  \right)
  \in U_L.
\end{equation*}

Consider a closed subvariety $X\subseteq \bP V$. As \cite[\S 6.1.1]{3264} makes clear, it is especially convenient, having fixed a basis $\cB=(e_i)\subset V$, to describe the Fano schemes \Cref{eq:fano} on open patches $U_L$ for $L$ spanned by subsets of $\cB$. The following simple remark will be of some help in that respect.

\begin{lemma}\label{le:cvr.grsm}
  Let $\cB=(e_i)_{i\in I}\subset V$ be a basis of a finite-dimensional vector space and set
  \begin{equation*}
    \forall\left(S\subseteq I\right)
    \quad:\quad
    U_S:=U_{L_S}
    \quad\text{and}\quad
    L_S:=\spn\left\{e_i\right\}_{i\in S}.
  \end{equation*}
  For $k\in \bZ_{\ge 0}$ we have
  \begin{equation*}
    \bG(k,\bP V)
    =
    \bigcup_{\substack{S\subseteq I\\|S|=\dim V-(k+1)}}U_{S}
  \end{equation*}
\end{lemma}
\begin{proof}
  Fix a $(k+1)$-dimensional subspace $W\le V$, and consider the sentence
  \begin{equation*}
    P_d
    \quad:\quad
    \forall\left(S\subseteq I,\ |S|=d\right)
    \left(W\cap L_S\ne \{0\}\right).
  \end{equation*}
  Assume $P_d$ valid for some $d\le \dim V-(k+1)$; we will argue that $P_{d-1}$ must then be valid as well, achieving a contradiction at $P_0$. 

  Let $S\subseteq I$ be a $(d-1)$-element set. If $W$ avoids $L_S$, then it must contain vectors in all
  \begin{equation*}
    L_{S\sqcup\{i\}}\setminus L_S
    ,\quad
    i\in I\setminus S.
  \end{equation*}
  Such vectors will span a subspace of $W$ of dimension
  \begin{equation*}
    |I\setminus S|
    =
    \dim V-d+1
    \ge
    k+2
    >
    k+1
    =
    \dim W.
  \end{equation*}
  This being absurd, we have shown that $W\cap L_S\ne \{0\}$ and hence verified $P_{d-1}$ (for the cardinality-$(d-1)$ subset $S\subseteq I$ was arbitrary). 
\end{proof}

\pf{th:stt.cj}
\begin{th:stt.cj}
  Before settling into proving \Cref{item:th:stt.cj:gen}, observe that the other claims are indeed simple consequences:
  \begin{itemize}[wide]
  \item The conditions of \Cref{eq:pvpc} in fact require that $\pi\in \cP^{\ge 2}_{[m]}$ and $\sharp\pi=d$, and if $2d>m$ then no such partition exists ($d$ parts, all of size $\ge 2$). This settles \Cref{item:th:stt.cj:d.lg}.

  \item As for \Cref{item:th:stt.cj:isol}, note that a point $V_{\pi\mathbf{c}}$ is isolated precisely when $\pi\in \cP^{2}_{[m]}$ (all parts are of size precisely 2), so as to afford no choice among the $c_i$ (up to simultaneous scaling across every part $\pi^i$). This forces $2d=m$, and the statement's description of the points is what \Cref{eq:pvpc} specializes to.
  \end{itemize}
  We henceforth focus on \Cref{item:th:stt.cj:gen} for the duration of the proof. \Cref{le:cvr.grsm} localizes the problem: it will suffice to describe $F_{d-1}(X)\cap U$ for
  \begin{equation*}
    U:=
    \left\{
      \bP W
      \in
      \bG(d-1,\bP V)
      \ :\
      W\cap L=\{0\}
    \right\}
    \subseteq
    \bG(d-1,\bP V)
  \end{equation*}
  for $(m-d)$-dimensional $L\le V$ obtained by annihilating $d$ of the $m$ coordinates $x_i$ fixed throughout. By permutation-invariance, moreover, we may as well take $L=\left\{x_i=0,\ 1\le i\le d\right\}$. Per \cite[\S 6.1.1]{3264}, $F_{d-1}(X)\cap U\subseteq U$ is described as a zero locus as follows\footnote{The matrix equation display \cite[\S 6.1.1]{3264} seems to be marred by a small typo: the last column should consist of entries $a_{\bullet,n}$ rather than $a_{\bullet,n+1}$.}:
  \begin{itemize}[wide]
  \item identify $U$ with the space of $d\times m$ matrices
    \begin{equation}\label{eq:ta}
      \begin{pmatrix}
        T:=
        I_d&\mid&A
      \end{pmatrix}
      =
      (t_{ij})_{i,j}
      \in M_{d,m}
      ,\quad
      A=(a_{ij})_{i,j}
      \in M_{d,m-d};
    \end{equation}

  \item with $f$ as in the statement (the $(m-1)^{st}$ elementary symmetric function), expand $f\left((s_i)\cdot T\right)=0$, the evaluation of $f$ on the $m$ entries of the row $(s_i)\cdot T$, as a polynomial in the indeterminates $s_i$, $1\le i\le d$. 
    
  \item the polynomials in the entries of $A$ appearing as coefficients in that expansion define $F_{d-1}(X)\cap U$. 
  \end{itemize}
  In this setup, the claim is that the elements of that zero locus are precisely those \Cref{eq:ta} of either of the two following types:
  \begin{enumerate}[(a),wide]
  \item some two columns of $T$ (hence $A$) vanishing identically (equivalent to at least \emph{one} column vanishing identically);

  \item\label{item:th:stt.cj:pf.nz} or, for some partition $\pi\in \cM^{\ge 1}_{[m]}$ of the column index set, we have 
    \begin{equation*}
      \forall \left(1\le i\le \sharp \pi\right)
      \exists\left(v_i\in \Bbbk^d\right)
      \exists\left(
        \mathbf{c}
        \in
        \left(\Bbbk^{\times}\right)^m\text{ as in \Cref{eq:pvpc}}
      \right)
      \quad:\quad
      \left(j\in \pi^i\xRightarrow{\quad}T_{\bullet j}=c_j v_i\right).
    \end{equation*}
  \end{enumerate}
  It will thus suffice to assume no zero columns and prove that \Cref{item:th:stt.cj:pf.nz} obtains. To that end, observe that the non-zero-columns assumption means that the linear combinations
  \begin{equation*}
    a_{1j}s_1+\cdots + a_{dj}s_d
    ,\quad
    1\le j\le m-d
  \end{equation*}
  are non-trivial linear maps in $s_i$, which by the infinitude of $\Bbbk$ we can take as indeterminates. The hypothesis (that \Cref{eq:ta} belongs to $F_{d-1}(X)\cap U$) can then be recast as the rational-function identity
  \begin{equation}\label{eq:rat.id}
    \sum_{j=1}^{m} \frac 1{t_{1j}s_1+\cdots + t_{dj}s_d}
    \left(
      =
      \sum_{j=1}^{m} \frac 1{\text{$j^{th}$ entry of $(s_i)\cdot T$}}
    \right)   
    \quad
    =
    \quad    
    0.
  \end{equation}
  The reciprocals of degree-1 homogeneous polynomials which are mutual non-scalar-multiples are linearly independent by \Cref{le:lin.ind.recs}, meaning that the left-hand side of \Cref{eq:rat.id} reads
  \begin{equation*}
    \sum_{i=1}^{\sharp \pi}\sum_{j\in \pi^i}\frac 1{c_j f_i(s_{\bullet})}
    \quad\text{for}\quad
    \pi\in \cP^{\ge 1}_{[m]}
    ,\ c_j\in \Bbbk^{\times}
  \end{equation*}
  and $\sharp \pi$ distinct linear forms $f_i$, and the selfsame \Cref{le:lin.ind.recs} also provides the constraint that all $\sum_{j\in \pi^i}\frac 1{c_j}$ vanish (for varying $i$). 
\end{th:stt.cj}

\begin{remark}\label{re:not.exp.dim}
  As \Cref{th:stt.cj} makes clear, the specific Fano schemes it is concerned with are rather far from exhibiting ``expected behavior'': the naive dimension count \cite[Proposition 6.4]{3264} for Fano schemes $F_{d-1}(X)$ with $X\subseteq \bP V$ cut out by a degree-$(m-1)$ polynomial yields
  \begin{align*}
    \dim F_{d-1}(X)
    &=
      \dim \bG(d-1,\bP V)-\tbinom{d+m-2}{d-1}\\
    &=
      d\left(\dim V-d\right)-\tbinom{d+m-2}{d-1}
      \xlongequal[\quad]{\ \dim V=m\ }
      d(m-d)-\tbinom{d+m-2}{d-1},
  \end{align*}
  which may well be negative. This is certainly so if $d\ge 3$ in the case $m=2d$ of interest in the original problem posed in \cite[\S 5]{2507.22138v1}.
\end{remark}

The following simple auxiliary observation is presumably self-evident; the proof being short, we include it for completeness. 

\begin{lemma}\label{le:lin.ind.recs}
  Let $f_i=f_i(s_j,\ 1\le j\le d)$ be non-zero linear forms in $d$ variables over a field $\Bbbk$ spanning distinct lines in $\Bbbk[s_j]$.

  The reciprocals $\frac 1{f_i}$ are linearly independent over $\Bbbk$. 
\end{lemma}
\begin{proof}
  We assume $\Bbbk$ infinite (for we can always pass to an extension), evaluating the indeterminates $s_j$ thereon. We may as well assume at least one $f_i$ involves $s:=s_1$. Evaluating all other $s_j$, $j\ne 1$ so as to ensure non-zero denominators and rescaling those $f_i$ which depend on $s_1$ so that they are monic in the latter, a linear dependence relation would read

  \begin{equation}\label{eq:csa}
    \sum_{\ell}\frac {c_{\ell}}{s+a_{\ell}} = C\in \Bbbk
    ,\quad
    c_{\ell}
    \ \&\
    \left(\text{distinct }a_{\ell}\right)
    \in \Bbbk.
  \end{equation}
  Were $c_1$ (say) non-zero, all terms of \Cref{eq:csa} would have \emph{valuation} \cite[\S VI.3.1, Definition 1]{bourb_comm-alg_en_1972} $\nu\ge 0$ in the valuation ring $\Bbbk[s]_{(s+c_1)}$ (\emph{localization} \cite[\S II.2.2, Remark 3 and \S VI.1.4, Corollary 2]{bourb_comm-alg_en_1972} at the prime ideal $(s+c_1)$), while
  \begin{equation*}
    \nu\left(\frac{c_{1}}{s+a_{1}}\right)=-1.
  \end{equation*}
  An application of $\nu$ thus \cite[\S VI.3.1, Proposition 1]{bourb_comm-alg_en_1972} turns \Cref{eq:csa} into the absurd $-1=0$. 
\end{proof}

\section{Differential-operator symbols as polynomial invariants}\label{se:pol.inv}

The present section is motivated by \cite[Question pre \S 5]{2507.22138v1}. That, in turn, flows out of that paper's material on star transforms and integral/differential operators, but for our purposes the purely representation-theoretic question can undercut some of the PDE-oriented preamble. The setup, briefly, is as follows.

\begin{itemize}[wide]
\item Consider (real, finite-dimensional) representations $\bG\circlearrowright V,W$ of a finite group, of which the latter is a \emph{permutation representation} \cite[\S 1.2(c)]{serre_rep_1977}: one obtained via a morphism $\bG\to S_m$ to the symmetric group on $m$ symbols upon fixing an identification $W\cong \bR^m$. 

\item Consider also a $\bG$-equivariant morphism $V\xrightarrow{U}W$, allowing us to pull back polynomial functions the codomain to those on the domain:
  \begin{equation*}
    S W^*
    \xrightarrow[\quad\text{$\bG$-equivariant}\quad]{\quad U^*:=\circ U\quad}
    S V^*
    ,\quad
    S\bullet:=\text{\emph{symmetric algebra} \cite[\S 11.5, post Proposition 33]{df_3e}}.
  \end{equation*}

  
\end{itemize}

While there is arguably some ambiguity in the question as formulated originally (e.g. ``producing all generators of the invariant ring'' might be amenable to several interpretations), one version of what might have been meant is presumably as follows:

\begin{question}\label{qu:all.perm.rep}
  Is it the case, for fixed $\bG\circlearrowright V$ and varying permutation representations $\bG\circlearrowright W$, that the images of the induced maps
  \begin{equation}\label{eq:fix2fix}
    \left(S W^*\right)^{S_m}
    \xrightarrow[\quad\text{$\bG$-equivariant}\quad]{\quad U^*:=\circ U\quad}
    \left(S V^*\right)^{\bG}
  \end{equation}
  between fixed-point subspaces generate codomain as an algebra?
\end{question}

In working with finite-group representations we do occasionally allow positive-characteristic fields, but always assume the characteristic coprime to the group order. This affords much good behavior: $\bG$ is then \emph{linearly reductive} \cite[\S 1.1, Definition 1.4]{fkm}, we have recourse to \emph{Reynolds operators} (\cite[\S 1.1, Definition 1.5]{fkm}, \cite[p.216]{zbMATH01097431}), the algebra of polynomial invariants is finitely generated \cite[Theorems 1.1 and 1.2]{zbMATH01097431}, etc.

\begin{remark}\label{re:perm.mat.rows}
  To translate the discussion above back into the matrix language employed in \cite{2507.22138v1}, note that choosing bases for $\varphi:\bG\circlearrowright V$ and $\rho:\bG\circlearrowright W$ so that in the latter's case all $\rho_s$, $s\in \bG$ are permutation matrices, the intertwiner $U$ will be expressible as a $\left(\dim W\right)\times \left(\dim V\right)$-matrix satisfying the intertwining condition
  \begin{equation*}
    \rho_s U = U\varphi_s
    ,\quad
    \forall s\in \bG:
  \end{equation*}
  cf. \cite[Definition 3.1]{2507.22138v1}. As in the latter, then, $U$ amounts essentially (insubstantial base-choice issues aside) to selecting a $\bG$-invariant set of vectors in $V^*$ (identifiable with the rows of $U$ regarded as a matrix).

  \Cref{qu:all.perm.rep}, then, amounts to this: is it the case that the algebra $\left(SV^*\right)^{\bG}$ is generated by polynomials of the form
  \begin{equation*}
    p\left(\left(f_i\right)_{i=1}^m\right)
    ,\quad
    p\text{ symmetric in $m$ variables}
  \end{equation*}
  for $\bG$-invariant sets $\left\{f_i\right\}_i$ of linear forms on $V$?
\end{remark}

In this phrasing, \Cref{qu:all.perm.rep} is known to have an affirmative answer (essentially, up to rephrasing), under appropriate conditions on the order $|\bG|$ in relation to the field's characteristic. 

\begin{theorem}\label{th:orbs.suff}
  If the characteristic of the field $\Bbbk$ is either 0 or larger than $|\bG|$ then for any finite-dimensional $\bG$-representation $\bG\circlearrowright V$ over $\Bbbk$ the images of the maps \Cref{eq:fix2fix} generate the $\Bbbk$-algebra $\left(SV^*\right)^{\bG}$.
\end{theorem}
\begin{proof}
  Under the hypothesized conditions \cite[Theorem 2.1]{zbMATH01097431} (also \cite[Theorem 3.1.10]{sm_pol-inv}) shows that $\left(SV^*\right)^{\bG}$ is generated as an algebra by what that source refers to as \emph{orbit Chern classes}. These are by definition (\cite[(2.1)]{zbMATH01097431} and subsequent discussion) symmetric polynomials in the elements of $\bG$-orbits in $V^*$, hence the conclusion.
\end{proof}

Note that the common codomain of \Cref{eq:fix2fix} is certainly not the \emph{union} of the maps' images.

\begin{example}\label{ex:z2.minus}
  Over the field $\Bbbk:=\bR$ take for the representation $\bG\circlearrowright V$ of $\bG:=\bZ/2$ the sum of two copies of the non-trivial character. The action on the polynomial ring $SV^*\cong \Bbbk[x,y]$ is by indeterminate negation, and in order to conclude it suffices to observe that the $\bZ/2$-invariant polynomial $xy$ (say) is not expressible as
  \begin{itemize}[wide]
  \item a symmetric polynomial in $m$ variables;

  \item applied evaluated at an $m$-element $\bZ/2$-invariant set of linear forms $\alpha x+\beta y$.
  \end{itemize}
  Indeed, a $\bZ/2$-invariant set of linear forms would have to be a union of ($k$, say) $\pm$ pairs. A quadratic symmetric polynomial $f$ in $2k$ variables, evaluated on such a set, would return
  \begin{equation*}
    f\bigg(
    \left\{\pm \left(\alpha_i x+\beta_i y\right),\ 1\le i\le k\right\}
    \bigg)
    \quad
    \in
    \quad
    \Bbbk \sum_{i=1}^k\left(\alpha_i x+\beta_i y\right)^2.
  \end{equation*}
  The latter line does not contain $xy$ over the reals or indeed, over any \emph{formally real} field, i.e. (\cite[Definition 15.2]{rjw_sq}, \cite[post Proposition 17.4]{lam_1st_2e_2001}) one in which non-trivial sums of squares do not vanish.
\end{example}



\addcontentsline{toc}{section}{References}

\def\polhk#1{\setbox0=\hbox{#1}{\ooalign{\hidewidth
  \lower1.5ex\hbox{`}\hidewidth\crcr\unhbox0}}}

\Addresses

\end{document}